\documentclass[11pt]{amsart}

\usepackage{amsfonts,amsmath, amsthm, amscd, amssymb, graphicx, color, mathrsfs, stmaryrd}
\usepackage{pb-diagram}
\usepackage{epic,eepic,epsfig}
\usepackage{latexsym}
\usepackage{xcolor}
\usepackage{xy}
\xyoption{all}

\newtheorem{Thm}{Theorem}[section]

\newtheorem{Cor}[Thm]{Corollary}

\newtheorem{Prop}[Thm]{Proposition}

\theoremstyle{definition}
\newtheorem{Def}[Thm]{Definition}

\theoremstyle{remark}

\def\eps{\varepsilon}

\def\Mdb{\mathbb M}
\def\Ndb{\mathbb N}

\def\Rdb{\mathbb R}

\def \diam{\text{diam }}
\def\Lip{\text{Lip}}

\def\n{\overline n}
\def\m{\overline m}

\newcommand{\bib}{\bibitem}

\begin{document}

\title[Asymptotic and coarse Lipschitz structures of Banach spaces]{Asymptotic and coarse Lipschitz structures of quasi-reflexive Banach spaces}

%\title{Asymptotic and coarse Lipschitz structures of quasi-reflexive Banach spaces }

\author{G. Lancien$^\clubsuit$}
\address{$\clubsuit$ Laboratoire de Math\'ematiques de Besan\c con, Universit\'e Bourgogne  Franche-Comt\'e, 16 route de Gray, 25030 Besan\c con C\'edex, France}
\email{gilles.lancien@univ-fcomte.fr}

\author{M. Raja$^\spadesuit$}
\address{$^\spadesuit$ Departamento de Matem\'{a}ticas, Universidad de Murcia, Campus de Espinardo, 30100 Espinardo, Murcia, Spain}
\email{matias@um.es}

\subjclass[2010]{46B80, 46B20}
%\date{\today}

\thanks{The first named author was supported by the French
``Investissements d'Avenir'' program, project ISITE-BFC (contract
 ANR-15-IDEX-03) and as a participant of the ``NSF Workshop in Analysis and Probability'' at Texas A\&M University.}
\thanks{The second named author was partially supported by the grants MINECO/FEDER MTM2014-57838-C2-1-P and Fundaci\'on S\'eneca CARM 19368/PI/14.}

\keywords{coarse Lipschitz embeddings, quasi-reflexive Banach spaces, asymptotic structure of Banach spaces}

\maketitle

\begin{abstract}  In this note, we extend to the setting of quasi-reflexive spaces a classical result of N. Kalton and L. Randrianarivony on the coarse Lipschitz structure of reflexive and asymptotically uniformly smooth Banach spaces. As an application, we show for instance, that for $1\le q<p$, a $q$-asymptotically uniformly convex Banach space does not coarse Lipschitz embed into a $p$-asymptotically uniformly smooth quasi-reflexive Banach space. This extends a recent result of B.M. Braga.
\end{abstract}

\section{Introduction.}
We start this note with some basic definitions on metric embeddings.\\
Let $(M,d)$ and $(N,\delta)$ be two metric spaces and $f$ be a map from $M$ into $N$. We define the {\it compression modulus} of $f$ by
$$\rho_f(t)=\inf\big\{\delta(f(x),f(y)),\  d(x,y)\geq t\big\},$$
and the {\it expansion modulus} of $f$ by
$$\omega_f(t)={\rm sup}\big\{\delta(f(x),f(y)),\  d(x,y)\leq t\big\}.$$
We say that $f$ is a {\it Lipschitz embedding} if there exist  $A,B$ in $(0,\infty)$ such that $\omega_f(t)\le Bt$ and $\rho_f(t)\ge At$.\\
If the metric space $M$ is unbounded, we say that $f$ is a {\it coarse embedding} if $\lim_{t\to\infty}\rho_f(t)=\infty$ and $\omega_f(t)<\infty$ for all $t>0$. Note that if $M$ is a Banach space, we have automatically in that case that $\omega_f$ is dominated by an affine function.\\
We say that $f$ is a {\it coarse Lipschitz embedding} if there exist  $A,B,C,D$ in $(0,+\infty)$ such that $\omega_f(t)\le Bt+D$ and $\rho_f(t)\ge At-C$.\\
In order to refine the scale of coarse embeddings, E. Guentner and J. Kaminker introduced in \cite{GuKa} the following notion. Let $X$ and $Y$ be two Banach spaces. We define $\alpha_Y(X)$ as the supremum of all $\alpha \in [0,1)$ for which there exists a coarse embedding $f:X\to Y$ and $A,C$ in $(0,\infty)$ so that $\rho_f(t)\ge At^\alpha-C$ for all $t>0$. Then, $\alpha_Y(X)$ is called the {\it compression exponent of $X$ in $Y$}.

\medskip We now turn to the definitions of the uniform asymptotic properties of norms that will be considered in this paper. For a Banach space $(X,\|\ \|)$ we
denote by $B_X$ the closed unit ball of $X$ and by $S_X$ its unit
sphere. The following definitions are due to V. Milman \cite{Milman} and we follow the notation from \cite{JohnsonLindenstraussPreissSchechtman2002}. For $t\in [0,\infty)$, $x\in S_X$ and $Y$ a closed linear subspace of $X$, we define
$$\overline{\rho}_X(t,x,Y)=\sup_{y\in S_Y}\big(\|x+t y\|-1\big)\ \ \ \ {\rm and}\ \ \
\ \overline{\delta}_X(t,x,Y)=\inf_{y\in S_Y}\big(\|x+t y\|-1\big).$$ Then
$$\overline{\rho}_X(t,x)=\inf_{{\rm
dim}(X/Y)<\infty}\overline{\rho}_X(t,x,Y)\ \ \ \ {\rm and}\ \ \
\ \overline{\delta}_X(t,x)=\sup_{{\rm
dim}(X/Y)<\infty}\overline{\delta}_X(t,x,Y)$$
and
$$\overline{\rho}_X(t)=\sup_{x\in S_X}\overline{\rho}_X(t,x)\ \ \ \ {\rm and}\ \ \ \
\overline{\delta}_X(t)=\inf_{x\in S_X}\overline{\delta}_X(t,x).$$
The norm $\|\ \|$ is said to be
{\it asymptotically uniformly smooth} (in short AUS) if
$$\lim_{t \to 0}\frac{\overline{\rho}_X(t)}{t}=0.$$
It is said to be {\it asymptotically uniformly convex} (in short
AUC) if $$\forall t>0 \ \ \ \ \overline{\delta}_X(t)>0.$$
Let $p\in (1,\infty)$ and $q\in [1,\infty)$.\\
We say that the norm of $X$ is {\it $p$-AUS} if there exists $c>0$ such that for all $t\in [0,\infty)$, $\overline{\rho}_X(t)\le ct^p$.\\
We say that the norm of $X$ is {\it $q$-AUC} if there exits $c>0$ such that for all $t \in [0,1]$, $\overline{\delta}_X(t)\ge ct^q$.\\
Similarly, there is in $X^*$ a modulus of weak$^*$
asymptotic uniform convexity defined by
$$ \overline{\delta}_X^*(t)=\inf_{x^*\in S_{X^*}}\sup_{E}\inf_{y^*\in S_E}\big(\|x^*+ty^*\|-1\big),$$
where $E$ runs through all weak$^*$-closed subspaces of
$X^*$ of finite codimension. The norm of $X^*$ is said to be {\it weak$^*$ uniformly asymptotically convex} (in short weak$^*$-AUC) if $\overline{\delta}_X^*(t)>0$ for all $t$ in $(0,\infty)$. If there exists $c>0$ and $q\in [1,\infty)$ such that for all $t\in [0,1]$ $\overline{\delta}_X^*(t)\ge ct^q$, we say that the norm of $X^*$ is $q$-weak$^*$-AUC.

\medskip Let us recall the following classical duality result concerning these moduli (see for instance \cite{DKLR} Corollary 2.3 for a precise statement).

\begin{Prop}\label{duality} Let $X$ be a Banach space.\\
Then $\|\ \|_X$ is AUS if and and only if $\|\ \|_{X^*}$ is weak$^*$-AUC.\\
If $p,q\in (1,\infty)$ are conjugate exponents, then $\|\ \|_X$ is $p$-AUS if and and only if $\|\ \|_{X^*}$ is $q$-weak$^*$-AUC.
\end{Prop}

\medskip The main purpose of this note is to extend to the quasi-reflexive case an important  result obtained by N. Kalton and L. Randrianarivony in \cite{KR} about coarse Lipschitz embeddings into $p$-AUS reflexive spaces. In order to explain their result, we need to introduce special metric graphs that we shall call  {\it Kalton-Randrianarivony's graphs}. For an infinite subset $\Mdb$ of $\Ndb$ and $k\in \Ndb$, we denote
$$G_k(\Mdb)=\{\n=(n_1,..,n_k),\ n_i\in\Mdb\ \ n_1<..<n_k\}.$$
Then we equip $G_k(\Mdb)$ with the distance $d(\n,\m)=|\{j,\ n_j\neq m_j\}|$. The fundamental result of their paper (Theorem 4.2 in \cite{KR}) can be rephrased as follows.

\begin{Thm}\label{KaRa} {\bf (Kalton-Randriarivony 2008)} Let $p\in (1,\infty)$ and assume that $Y$ is a reflexive $p$-AUS Banach space. Then there exists a constant $C>0$ such that for any infinite subset $\Mdb$ of $\Ndb$, any $f: (G_k(\Mdb),d) \to Y$ Lipschitz map and any $\eps>0$, there exists an infinite subset $\Mdb'$ of $\Mdb$, such that
$$\diam f(G_k(\Mdb')) \le CLip(f) k^{1/p}+\eps.$$
\end{Thm}
We refer the reader to \cite{KR} and \cite{K2} for the various applications derived by the authors. Very recently, B.M. Braga used the above theorem in \cite{Br} to develop many other applications. We will only mention one of them in this introduction (Corollary 4.5 in \cite{Br}).

\begin{Thm}\label{Braga} {\bf (Braga 2016)} Let $1\le q<p$, $X$ be a $q$-AUC Banach space and $Y$ be a $p$-AUS reflexive Banach space. Then $\alpha_Y(X)\le q/p$.
\end{Thm}

We recall that a Banach space is said to be {\it quasi-reflexive} if the image of its canonical embedding into its bidual is of finite codimension in its bidual. The aim of this note is to obtain a version of the above Theorem \ref{KaRa} for quasi-reflexive Banach spaces. This is done in section 2 and our main results are Theorem \ref{QRAUS} and Theorem \ref{QRAUS+}. In section 3, we apply them in order to extend some results from \cite{Br} to the quasi-reflexive setting, including the above Theorem \ref{Braga}. Our applications are stated in Corollary \ref{ABS}, Theorem \ref{qcoBS}, Corollary \ref{AUC-AUS} and Corollary \ref{Sz}.

\section{The main result.}

We first need the following simple property of the bidual of a Banach space $X$, in relation with the modulus of uniform asymptotic smoothness of $X$.
\begin{Prop}\label{waus}
Let $X$ be a Banach space. Then the bidual norm on $X^{**}$ has the following property. For any $t\in (0,1)$, any weak$^*$-null sequence $(x^{**}_n)_{n=1}^\infty$ in $B_{X^{**}}$ and any  $x \in S_X$ we have:
$$ \limsup_{n \rightarrow \infty} \| x+tx^{**}_n \| \leq 1 + \overline{\rho}_{X}(t,x). $$
\end{Prop}

\noindent
\begin{proof} Let $x$ in $S_X$ and $\eps>0$. By definition, there exists a finite codimensional subspace $Y$ of $X$ such that
\begin{equation}\label{applyaus}
x + t B_Y \subset (1+ \overline{\rho}_{X}(t,x) + \varepsilon) B_X.
\end{equation}
There exist $x_1^*,..,x_k^*\in X^*$ such that
$$Y=\bigcap_{i=1}^k \{x \in X,\ x_i^*(x)=0\}.$$ Note that the weak$^*$-closure of $Y$ in $X^{**}$ is
$$\overline{Y}^{w^*} = Y^{\perp \perp}=\bigcap_{i=1}^k \{x^{**} \in X^{**},\  x^{**}(x_i^*)=0\}.$$
By Goldstine's Theorem we have that the weak$^*$-closures in $X^{**}$ of $B_Y$ and $B_X$ are respectively $B_{\overline{Y}^{w^*}}$ and $B_{X^{**}}$. So, taking the weak$^*$-closures in (\ref{applyaus}) we get that
$$x+t B_{\overline{Y}^{w^*}} \subset (1+ \overline{\rho}_{X}(t,x) + \varepsilon) B_{X^{**}}.$$
Since the sequence $(x^{**}_n)_{n=1}^\infty$ is weak$^*$-null, we have that for all $i$ in $\{1,..,k\}$, $\lim_{n\to \infty}x^{**}_n(x_i^*)=0$. It follows easily that $\lim_{n\to \infty}d(x^{**}_n,B_{\overline{Y}^{w^*}})=0$. We deduce, that for $n$ large enough
$$x+tx^{**}_n\in (1+ \overline{\rho}_{X}(t) + 2\varepsilon) B_{X^{**}},$$
which concludes our proof.
\end{proof}

%\noindent{\bf Remark.} Note that in the above result, $x\in S_X$ cannot in general be replaced by $x^{**}\in S_{X^{**}}$. Indeed, it follows easily from Goldstine's theorem that a bidual norm cannot be weak$^*$-AUC, unless the space is reflexive. This implies, using Proposition \ref{duality} that a dual norm (so in particular a bidual norm) cannot be AUS unless the space is reflexive.

We shall now give an analogue of Theorem \ref{KaRa} when $Y$ is only assumed to be quasi-reflexive and $p$-AUS for some $p\in (1,\infty)$. The idea is to adapt techniques from a work by F. N\'{e}tillard \cite{Ne} on the coarse Lipschitz embeddings between James spaces. To this end, for $\Mdb$ an infinite subset of $\Ndb$, we denote $I_k(\Mdb)$ the set of strictly interlaced pairs in $G_k(\Mdb)$, namely :
$$I_k(\Mdb)=\big\{(\n,\m) \in G_k(\Mdb)\times G_k(\Mdb),\ n_1 < m_1 < n_2 < m_2 < ... < n_k < m_k\big\}.$$
Note that for $(\n,\m) \in I_k(\Mdb)$, $d(\n,\m)=k$. Our statement is then the following.

\begin{Thm}\label{QRAUS} Let $p\in (1,\infty)$ and $Y$ be a  quasi-reflexive $p$-AUS Banach space. Then there exists a constant $C>0$ such that for any infinite subset $\Mdb$ of $\Ndb$, any $f:(G_k(\Mdb),d) \to Y^{**}$ Lipschitz and any $\eps>0$ there exists an infinite subset $\Mdb'$ of $\Mdb$, such that
$$\forall (\n,\m) \in I_k(\Mdb')\ \ \ \|f(\n) - f(\m)\| \leq C Lip(f)k^{\frac{1}{p}} + \eps.$$
\end{Thm}

In fact, we will show a more general result, which follows the ideas of section 6 in \cite{KR}. Before to state it, we need some preparation. We briefly recall the setting of section 6 in \cite{KR}.\\
So let $Y$ be a Banach space and denote by $\overline \rho _Y$ its
modulus of asymptotic uniform smoothness. It is easily checked
that $\overline \rho _Y$ is an Orlicz function. Then we define the
Orlicz sequence space:
$$\ell_{\overline \rho _Y}=\Big\{x\in \Rdb^\Ndb,\ \exists r>0\ \
\sum_{n=1}^\infty \overline \rho _Y\big(\frac{|x_n|}{r}\big)<\infty\Big\},$$
equipped with the norm
$$\|x\|_{\overline \rho _Y}=\inf\Big\{r>0,\ \sum_{n=1}^\infty \overline \rho _Y\big(\frac{|x_n|}{r}\big)\le 1\Big\}.$$
Next we construct a sequence of norms $(N_k)_{k=1}^\infty$, where $N_k$ is a norm on $\Rdb^k$, as follows.\\
For all $\xi\in \Rdb$, $N_1(\xi)=|\xi|$.\\
$N_2(\xi,\eta)=|\eta|$ if $\xi=0$ and
$$N_2(\xi,\eta)=|\xi|\big(1+\overline \rho _Y\big(\frac{|\eta|}{|\xi|}\big)\big)\ \ {\rm if}\ \xi\neq 0.$$
Then, for $k\ge 3$, we define by induction the following norm on $\Rdb^k$: $$N_k(\xi_1,..,\xi_k)=N_2\big(N_{k-1}(\xi_1,..,\xi_{k-1}),\xi_k\big).$$

\medskip The following property is proved in \cite{KR}.

\begin{Prop}\label{KaltonNorm} For any $k\in \Ndb$ and any $a\in \Rdb^k$:
$$N_k(a)\le e\|a\|_{\overline \rho _Y}.$$
\end{Prop}

\noindent Fix now $a=(a_1,..,a_k)$ a sequence of non zero real numbers and define the
following distance on $G_k(\Mdb)$, for $\Mdb$ infinite subset of $\Ndb$:
$$\forall \n,\m \in G_k(\Mdb),\ \ d_a(\n,\m)=\sum_{j,\ n_j\neq m_j} |a_j|.$$

We can now state our general result, from which Theorem \ref{QRAUS} is easily deduced.

\begin{Thm}\label{QRAUS+} Let $Y$ be a  quasi-reflexive Banach space, $a=(a_1,..,a_k)$ a sequence of non zero real numbers, $\Mdb$ an infinite subset of $\Ndb$ and let \\
$f:(G_k(\Mdb),d_a)\to Y^{**}$ be a Lipschitz map.
\\Then for any $\eps>0$ there exists an infinite subset $\Mdb'$ of $\Mdb$, such that
$$\forall (\n,\m) \in I_k(\Mdb')\ \ \ \|f(\n) - f(\m)\| \leq 2e Lip(f)\|a\|_{\overline \rho _Y} +\eps.$$
\end{Thm}
\begin{proof} Under the assumptions of Theorem \ref{QRAUS+}, we will show that  there exists an infinite subset $\Mdb'$ of $\Mdb$, such that
$$\forall (\n,\m) \in I_k(\Mdb')\ \ \ \|f(\n) - f(\m)\| \leq 2 \Lip(f)N_k(a) +\eps.$$
Then, the conclusion will follow from Proposition \ref{KaltonNorm}.

Since the graphs $G_k(\Ndb)$ are all countable, we may assume that $Y$ is separable. We can write $Y^{**}=Y\oplus E$, where $E$ is finite dimensional.

We will prove our statement by induction on $k \in \Ndb$. It is clearly true for $k=1$, so assume it is true for $k\in \Ndb$ and let $a=(a_1,..,a_{k+1})$ be a sequence of non zero real numbers and $f:(G_{k+1}(\Mdb),d_a)\to Y^{**}$ be a Lipschitz map. Let $\eps>0$ and fix $\eta \in (0,\frac{\eps}{2})$ (our initial choice of a small $\eta$ will be made precise later).\\
Since $Y$ is separable and quasi-reflexive, $Y^{*}$ is also separable. So, using weak$^*$-compactness in $Y^{**}$, we can find an infinite subset $\Mdb_0$ of $\Mdb$ such that
$$\forall \n \in G_{k}(\Mdb_0)\ \ \ \  w^*-\lim_{n_k\in \Mdb_0}f(\n,n_k)=g(\n)\in Y^{**}.$$
Using the weak$^*$ lower semi continuity of $\|\ \|_{Y^{**}}$ we get that the map\\
$g:G_{k}(\Mdb_0)\to Y^{**}$ satisfies $\Lip(g)\le \Lip(f)$. For $\n\in G_{k}(\Mdb_0)$, we can write $g(\n)=h(\n)+e(\n)$, with $h(\n)\in Y$ and $e(\n)\in E$. It then follows from Ramsey's theorem and the norm compactness of bounded sets in $E$ that there exists an infinite subset $\Mdb_1$ of $\Mdb_0$ such that
$$\forall \n,\m \in G_k(\Mdb_1)\ \ \|e(\n)-e(\m)\|\le \eta.$$
For $\n,\m \in G_{k}(\Mdb_1)$ and $t,l\in \Mdb_1$, set
$$u_{\n,\m,t,l} = \big(f(\n,t) - g(\n)\big) - \big(f(\m,l)-g(\m)\big).$$
Since $\|\ \|_{Y^{**}}$ is weak$^*$ lower semi continuous, we have that $$\|u_{\n,\m,t,l}\|\le 2\Lip(f)|a_{k+1}|.$$
On the other hand, it follows from our induction hypothesis that there exists an infinite subset $\Mdb_2$ of $\Mdb_1$ such that
$$\forall (\n,\m)\in I_k(\Mdb_2)\ \ \ \|g(\n)-g(\m)\|\le 2\Lip(f)N_k(a_1,..,a_k)+\eta.$$
Therefore
$$\forall (\n,\m)\in I_k(\Mdb_2)\ \ \ \|h(\n)-h(\m)\|\le 2\Lip(f)N_k(a_1,..,a_k)+2\eta.$$
Assume first that $h(\n)\neq h(\m)$. Then it follows from Proposition  \ref{waus} and the fact that $u_{\n,\m,t,l}$ is tending to $0$ in the weak$^*$ topology, as $t,l$ tend to $\infty$, that there exists $N_{(\n,\m)}\in \Mdb_2$ such that for all $t,l\in \Mdb_2$ satisfying $t,l\ge N_{(\n,\m)}$:
$$\|h(\n)-h(\m)+u_{\n,\m,t,l}\|\le \|h(\n)-h(\m)\|
\Big(1+\overline{\rho}_Y\big(\frac{2\Lip(f)|a_{k+1}|}{\|h(\n)-h(\m)\|}\big)\Big)+\eta$$
$$\le N_2\big(\|h(\n)-h(\m)\|,2\Lip(f)|a_{k+1}|\big)+\eta.$$
Note that if $h(\n)=h(\m)$, the above inequality is clearly true for all $t,l \in \Mdb_2$.\\
Therefore, we have that for all $(\n,\m)\in I_k(\Mdb_2)$ there exists $N_{(\n,\m)}\in \Mdb_2$ such that for all $t,l\ge N_{(\n,\m)}$:
$$\|h(\n)-h(\m)+u_{\n,\m,t,l}\|\le N_2\big(2\Lip(f)N_k(a_1,..,a_k)+2\eta,2\Lip(f)|a_{k+1}|\big)+\eta$$
$$\le 2\Lip(f)N_{k+1}(a_1,..,a_{k+1})+\frac{\eps}{2},$$
if $\eta$ was initially chosen small enough.\\
So we have proved that for all $(\n,\m)\in I_k(\Mdb_2)$ there exist $N_{(\n,\m)}\in \Mdb_2$ such that for all $t,l\ge N_{(\n,\m)}$:

\begin{equation}\label{eq1}
\|f(\n,t)-f(\m,l)\|\le 2\Lip(f)N_{k+1}(a_1,..,a_{k+1})+\eps.
\end{equation}

We now wish to construct $\Mdb'$ infinite subset of $\Mdb_2$ satisfying our conclusion. For that purpose, for a finite subset $F$ of $\Ndb$ of cardinality at least $2k$, we denote $I_k(F)$ the set of all $(\n,\m)\in I_k(\Ndb)$ such that $n_1<m_1<..<n_k<m_k\in F$.\\
Assume now $\Mdb_2=\{n_1<..<n_j<..\}$. We shall define inductively the elements of our set $\Mdb'=\{m_1<...<m_j<..\}$.\\
First, we set $m_1=n_1,...,m_{2k}=n_{2k}$. Then, for $j>2k$, we define $m_j=n_{\phi(j)}>m_{j-1}$ in such a way that for all $(\n,\m)\in I_k(\{m_1,..,m_{j-1}\})$ we have that $m_j\ge N_{(\n,\m)}$. Then, it should be clear from equation (\ref{eq1}) and the construction of $\Mdb'$ that
$$\forall (\n,\m) \in I_{k+1}(\Mdb')\ \ \ \|f(\n) - f(\m)\| \leq 2\Lip(f)N_{k+1}(a_1,..,a_{k+1}) + \eps.$$
This finishes our induction.
\end{proof}

\noindent{\bf Remark.} Suppose now that $Y$ is a non reflexive Banach space and fix $\theta \in (0,1)$. Then, James' Theorem (see \cite{James}) insures the existence of a sequence
$(x_n)_{n=1}^\infty$ in $S_X$ and a sequence $(x_n^*)_{n=1}^\infty$ in $S_{X^*}$ such that
$$x_n^*(x_i)=\theta\ {\rm if}\ n\le i\ \ {\rm and}\ \ x_n^*(x_i)=0\ {\rm if}\ n>i.$$
In particular, for all $n_1<..<n_k<m_1<..<m_k$:
\begin{equation}\label{james}
\|x_{n_1}+..+x_{n_k}-(x_{m_1}+..+x_{m_k})\|\ge \theta k.
\end{equation}
Define now, for $k\in \Ndb$ and $\n\in G_k(\Ndb)$, $f(\n)=x_{n_1}+..+x_{n_k}$. Then $f$ is clearly 2-Lipschitz. On the other hand, it follows from (\ref{james}) that for any infinite subset $\Mdb$ of $\Ndb$, $\diam \big(f(G_k(\Mdb))\big)\ge \theta k$. This shows that the conclusion of Theorem \ref{KaRa} cannot hold as soon as $Y$ is non reflexive. This obstacle is overcome in our Theorems \ref{QRAUS} and \ref{QRAUS+} by considering particular pairs of elements in $G_k(\Ndb)$ that are $k$-separated, namely the strictly interlaced pairs from $I_k(\Ndb)$. Note also that the pairs considered in our above application of James' Theorem are $k$-separated but at the ``opposite'' of being interlaced, since $n_1<..<n_k<m_1<..<m_k$.

\section{Applications.}

Let us start with the following definitions.

\begin{Def} \

(i) A Banach space $X$ has the {\it Banach-Saks property} if every bounded sequence in $X$ admits a subsequence whose Ces\`{a}ro means converge in norm.

(ii) A Banach space $X$ has the {\it alternating Banach-Saks property} if for every bounded sequence $(x_n)_{n=1}^\infty$ in $X$, there exists a subsequence $(x_{n_k})_{k=1}^\infty$ of $(x_n)_{n=1}^\infty$ and a sequence $(\eps_k)_{k=1}^\infty \in \{-1,1\}^\Ndb$ such that the Ces\`{a}ro means of the sequence $(\eps_kx_{n_k})_{k=1}^\infty$ converge in norm.

\end{Def}

Our last remark of section 2 was used in \cite{BKL} (Theorem 4.1), to show that if a Banach space coarse Lipschitz embeds into a reflexive AUS Banach space, then $X$ is reflexive. This result, was recently improved by B.M Braga \cite{Br} (Theorem 1.3) who showed that actually $X$ must have the Banach-Saks property (which clearly implies reflexivity). As a first application of our result we obtain.

\begin{Prop}\label{ABS} Assume that $X$ is a Banach space which coarse Lipschitz embeds into a quasi reflexive AUS Banach space $Y$. Then $X$ has the alternating Banach-Saks property.
\end{Prop}

\begin{proof} Since $Y$ is AUS, there exists $p\in (1,\infty)$ such that $Y$ is p-AUS. This a consequence of a result of Knaust, Odell and Schlumprecht \cite{KOS} in the separable case and of the second named author for the general case \cite{Ra}.\\
Assume also that $X$ does not have the alternating Banach-Saks property. Then it follows from the work of B. Beauzamy (Theorem II.2 in \cite{Be}) that there exists a sequence $(x_n)_{n=1}^\infty$ in $X$ such that for all $k\in \Ndb$, all $\eps_1,..,\eps_k \in \{-1,1\}$ and all $n_1<..<n_k$:
\begin{equation}\label{notABS}
\frac12 \le \Big\|\frac1k \sum_{i=1}^k \eps_i x_{n_i}\Big\|\le \frac32.
\end{equation}
Assume now that $X$ coarse Lipschitz embeds into $Y$. Then, after a linear change of variable if necessary, there exists $f:X\to Y$ and $A,B>0$  such that
$$\forall x,x'\in X\ \ \ \|x-x'\|\ge \frac12\Rightarrow A\|x-x'\| \le \|f(x)-f(x')\|\le B\|x-x'\|.$$
We then define $\varphi_k: (G_k(\Ndb),d) \to X$ as follows:
$$\varphi_k(\n)=x_{n_1}+..+x_{n_k},\ \ \ {\rm for}\ \ \n=(n_1,..,n_k)\in G_k(\Ndb).$$
We clearly have that $\Lip(\varphi_k)\le 3$. Moreover, for $\n\neq \m \in G_k(\Ndb)$ we have $\|\varphi_k(\n)-\varphi_k(\m)\|\ge 1$. We deduce that for all $k\in \Ndb$, $\Lip(f \circ \varphi_k)\le 3B$.\\
It now follows from Theorem \ref{QRAUS} that there exists a constant $C>0$ such that for all $k\in \Ndb$, there is an infinite subset $\Mdb_k$ of $\Ndb$ so that
$$\forall (\n,\m)\in I_k(\Mdb_k)\ \ \|(f \circ \varphi_k)(\n)-(f \circ \varphi_k)(\n)\|\le Ck^{1/p}.$$
On the other hand, it follows from (\ref{notABS}) that for all $(\n,\m)\in I_k(\Mdb_k)$, we have: $\|\varphi_k(\n)-\varphi_k(\m)\|\ge k.$ Therefore
$$\forall (\n,\m)\in I_k(\Mdb_k)\ \ \|(f \circ \varphi_k)(\n)-(f \circ \varphi_k)(\m)\| \ge Ak.$$
This yields a contradiction for $k$ large enough.

\end{proof}

In order to state some more quantitative results, we will need the following definition.

\begin{Def} Let $q\in (1,\infty)$ and $X$ be a Banach space. We say that $X$ has the $q$-co-Banach-Saks property if for every semi-normalized weakly null sequence $(x_n)_{n=1}^\infty$ in $X$ there exists a subsequence $(x_{n_j})_{j=1}^\infty$ and $c>0$ such that for all $k\in \Ndb$ and all $k\le n_1<..<n_k$:
$$\|x_{n_1}+...+x_{n_k}\|\ge ck^{1/q}.$$
\end{Def}

For the proof of the following result, we refer the reader to Proposition 4.6 in \cite{K2}, or Proposition 2.3 in \cite{Br} and references therein.

\begin{Prop}\label{coBSAUC} Let $q\in (1,\infty)$ and $X$ be a Banach space. If $X$ is $q$-AUC, then $X$ has the $q$-co-Banach-Saks property.
\end{Prop}

Using our Theorem \ref{QRAUS} and adapting the arguments of Braga in \cite{Br} we obtain.

\begin{Thm}\label{qcoBS} Let $1<q<p$ in $(1,\infty)$. Assume that $X$ is an infinite dimensional Banach space with the $q$-co-Banach-Saks property and $Y$ is a $p$-AUS and quasi reflexive Banach space. Then $X$ does not coarse Lipschitz embed into $Y$. More precisely, the compression exponent $\alpha_Y(X)$ of $X$ into $Y$ satisfies the following.\\
(i) If $X$ contains an isomorphic copy of $\ell_1$, then $\alpha_Y(X)\le \frac{1}{p}$.\\
(ii) Otherwise, $\alpha_Y(X)\le \frac{q}{p}$.
\end{Thm}

\begin{proof} We follow the proof of Theorem 4.1 in \cite{Br}.

\smallskip
Assume first that $X$ contains an isomorphic copy of $\ell_1$ and that $\alpha_Y(X)> \frac{1}{p}$. Then there exists $f:X\to Y$ and $\theta,A,B>0$ and $\alpha>\frac{1}{p}$ such that
$$\forall x,x'\in X\ \ \ \|x-x'\|\ge \theta\Rightarrow A\|x-x'\|^\alpha \le \|f(x)-f(x')\|\le B\|x-x'\|.$$
Let $(x_n)_{n=1}^\infty$ be a sequence in $X$ which is equivalent to the canonical basis of $\ell_1$. We may assume, after a dilation, that there exists $K\ge 1$ such that
$$\forall a_1,..,a_k \in \Rdb,\ \ \ \theta\sum_{i=1}^k |a_i|\le \Big\|\sum_{i=1}^k a_i x_i\Big\| \le K\theta\sum_{i=1}^k |a_i|.$$
Then define $\varphi_k:(G_k(\Ndb),d)\to X$ by
$$\varphi_k(\n)=x_{n_1}+..+x_{n_k},\ \ \ {\rm for}\ \ \n=(n_1,..,n_k)\in G_k(\Ndb).$$
We clearly have that $\Lip(\varphi_k)\le 2K\theta$. Moreover, for $\n\neq \m \in G_k(\Ndb)$ we have $\|\varphi_k(\n)-\varphi_k(\m)\|\ge 2\theta$. We deduce that for all $k\in \Ndb$, $\Lip(f \circ \varphi_k)\le 2\theta KB$.\\
It now follows from Theorem \ref{QRAUS} that there exists a constant $C>0$ such that for all $k\in \Ndb$, there is an infinite subset $\Mdb_k$ of $\Ndb$ so that
$$\forall (\n,\m)\in I_k(\Mdb_k)\ \ \|(f \circ \varphi_k)(\n)-(f \circ \varphi_k)(\n)\|\le Ck^{1/p}.$$
On the other hand, for all $(\n,\m)\in I_k(\Mdb_k)$, $\|\varphi_k(\n)-\varphi_k(\m)\|\ge 2k\theta.$ Therefore
$$\forall (\n,\m)\in I_k(\Mdb_k)\ \ \|(f \circ \varphi_k)(\n)-(f \circ \varphi_k)(\m)\| \ge A2^\alpha\theta^\alpha k^\alpha.$$
This yields a contradiction if $k$ was chosen large enough.

\medskip
Assume now that $X$ does not contain an isomorphic copy of $\ell_1$ and that $\alpha_Y(X)> \frac{q}{p}$. Then there exists $f:X\to Y$,  $\alpha>\frac{q}{p}$ and $\theta,A,B>0$ such that
$$\forall x,x'\in X\ \ \ \|x-x'\|\ge \theta\Rightarrow A\|x-x'\|^\alpha \le \|f(x)-f(x')\|\le B\|x-x'\|.$$
Now, by Rosenthal's theorem, we can pick a normalized weakly null sequence $(x_n)_{n=1}^\infty$ in $X$. By extracting a subsequence, we may also assume that $(x_n)_{n=1}^\infty$ is a 2-basic sequence in $X$. We then define $\varphi_k: (G_k(\Ndb),d) \to X$ as follows:
$$\varphi_k(\n)=2\theta(x_{n_1}+..+x_{n_k}),\ \ \ {\rm for}\ \ \n=(n_1,..,n_k)\in G_k(\Ndb).$$
Note that $\Lip(\varphi_k)\le 4\theta$. Since $(x_n)_{n=1}^\infty$ is 2-basic, we have that for all $\n\neq \m \in G_k(\Ndb)$, $\|\varphi_k(\n)-\varphi_k(\m)\|\ge \theta$. Therefore $\Lip(f \circ \varphi_k)\le 4\theta B$, for all $k\in \Ndb$.\\
It now follows from Theorem \ref{QRAUS} that, for any $k\in \Ndb$, there exists an infinite subset $\Mdb_k$ of $\Ndb$ such that
\begin{equation}\label{interlaced}
\forall (\n,\m)\in I_k(\Mdb_k)\ \ \|(f \circ \varphi_k)(\n)-(f \circ \varphi_k)(\m)\|\le Ck^{1/p},
\end{equation}
where $C>0$ is a constant independent of $k$.\\
We can make sure in our construction that for all $k\in \Ndb$, $\Mdb_{k+1}\subset \Mdb_k$. Let  now $\Mdb$ be the diagonalization of the sequence $(\Mdb_k)_{k=1}^\infty$ and enumerate $\Mdb=\{r_1<..<r_i<..\}$. Denote $z_n=x_{r_{2n}}-x_{r_{2n+1}}$. By applying the $q$-co-Banach-Saks property to the semi normalized weakly null sequence $(z_n)_{n=1}^\infty$, we can find a subsequence $(z_{n_j})_{j=1}^\infty$ of $(z_n)_{n=1}^\infty$ and a constant $d>0$ such that for all $k\in \Ndb$ and all $k\le n_1<..<n_k$:
$$\|\sum_{j=1}^k z_{n_j}\|=\|\sum_{j=1}^k (x_{r_{2n_j}}-x_{r_{2n_j+1}})\|$$
$$=\|\varphi_k(r_{2n_1},..,r_{2n_{k}})-\varphi_k(r_{2n_1+1},..,r_{2n_k+1})\|\ge dk^{1/q}.$$
If $k$ is chosen large enough so that $dk^{1/q}\ge \theta$, we get that for all\\ 
$k\le n_1<..<n_k$:
$$\|(f\circ \varphi_k)(r_{2n_1},..,r_{2n_{k}})-(f\circ \varphi_k)(r_{2n_1+1},..,r_{2n_k+1})\|\ge Ad^\alpha k^{\alpha/q}.$$
It is now important to note that, due to the diagonal construction of $\Mdb$, we have that  for $k\le n_1<..<n_k$, the pair $\big((r_{2n_1},..,r_{2n_{k}}),(r_{2n_1+1},..,r_{2n_k+1})\big)$ is an element of $I_k(\Mdb_k)$. Therefore, this yields a contradiction with (\ref{interlaced}) if $k$ is chosen large enough.
\end{proof}

The following result is now a direct consequence of Theorem \ref{qcoBS} and Proposition \ref{coBSAUC}.

\begin{Cor}\label{AUC-AUS} Let $1<q<p<\infty$. Assume that $X$ is $q$-AUC and $Y$ is a $p$-AUS and quasi reflexive Banach space. Then
$$\alpha_Y(X)\le \frac{p}{q}.$$
\end{Cor}

For our next application, we need to recall the definition of the Szlenk index. This ordinal index  was first introduced by W. Szlenk \cite{Szlenk1968}, in a slightly different form, in order to prove that there is no separable reflexive Banach space universal for the class of all separable reflexive Banach spaces.\\
So, let $X$ be a Banach space, $K$ a weak$^*$-compact subset of its dual $X^*$ and $\eps>0$. Then we define
$$s_\eps'(K)=\{ x^* \in K,\ {\rm for\ any}\ {\rm weak^*-neighborhood}\ U\ {\rm of}\ x^*,\  {\rm diam}(K \cap U) \ge \varepsilon\}$$
and inductively the sets $s_\eps^\alpha(K)$ for $\alpha$ ordinal as follows: $s_\eps^{\alpha+1}(K)=s_\eps'(s_\eps^\alpha(K))$ and $s_\eps^\alpha(K)=\bigcap_{\beta<\alpha}s_\eps^\beta(K)$ if $\alpha$ is a limit ordinal.\\
Then $Sz(K,\eps)=\inf\{\alpha,\ s^\alpha_\eps(K)=\emptyset\}$ if it exists and we denote $Sz(K,\eps)=\infty$ otherwise. Next we define $Sz(K)=\sup_{\eps>0}Sz(K,\eps)$.
The Szlenk index of $X$ is $Sz(X)=Sz(B_{X^*})$. We also denote $Sz(X,\eps)=Sz(B_{X^*},\eps)$. \\
We shall apply the following renorming theorem, which is proved in \cite{GKL2001}.

\begin{Thm}\label{GKL} Let $q\in (1,\infty)$ and $X$ be a Banach space. Assume that there exists $C>0$ such that
$$\forall \eps >0\ \ \ Sz(X,\eps)\le C\eps^{-q}.$$
Then, for all $r\in (q,\infty)$, $X$ admits an equivalent norm whose dual norm is $r$-weak$^*$-AUC.
\end{Thm}

Since a dual norm which is $r$-weak$^*$-AUC is also $r$-AUC, we obtain the following statement as an immediate consequence of our Theorem \ref{GKL} and Corollary \ref{AUC-AUS}.

\begin{Cor}\label{Sz} Let $1<q<p<\infty$. Assume that $Y$ is a $p$-AUS and quasi reflexive Banach space. Assume also that there exists $C>0$ such that
$$\forall \eps >0\ \ \ Sz(X,\eps)\le C\eps^{-q}.$$
Then
$$\alpha_{Y}(X^*)\le \frac{p}{q}.$$
\end{Cor}

\noindent {\bf Aknowledgements.}\\
The authors wish to thank F. Baudier and Th. Schlumprecht for pointing out the application to the alternating Banach-Saks property. The first named author also wants  to thank for their hospitality the Universidad de Murcia and the University of Texas A\&M, where part of this work was completed.

\end{document}